\documentclass[11pt]{amsart}

\usepackage[margin=1.25in]{geometry}

\usepackage{amsmath,amssymb,amsthm,amsfonts,amscd,flafter,
graphicx,verbatim,pinlabel,mathrsfs,setspace}
\usepackage[none]{hyphenat}
\usepackage[all]{xy}
\usepackage{array,booktabs,longtable}
\usepackage{etoolbox,mathtools}
\usepackage{epstopdf}
\usepackage[colorlinks=true,citecolor=blue,linkcolor=blue]{hyperref}
\usepackage[all]{hypcap}

\title{Obstructions to Lagrangian concordance}

\author{Christopher R.\ Cornwell}
\address{CIRGET \\ Universit\'e du Qu\'ebec \`a Montr\'eal}
\email{cornwell@cirget.ca}

\author{Lenhard Ng}
\address{Department of Mathematics \\ Duke University}
\email{ng@math.duke.edu}

\author{Steven Sivek}
\address{Department of Mathematics \\ Princeton University}
\email{ssivek@math.princeton.edu}

\def\R{{\mathbb{R}}}

\def\Q{{\mathbb{Q}}}

\newcommand\zt{\mathbb{F}}
\newcommand\zz{\mathbb{Z}}

\newcommand\ssm{\smallsetminus}

\newcommand\isomto{\xrightarrow{\sim}}

\newcommand{\tb}{\operatorname{tb}}
\newcommand{\maxtb}{\overline{\tb}}
\newcommand{\mindeg}{\operatorname{min-deg}}
\newcommand{\maxdeg}{\operatorname{max-deg}}
\newcommand{\Lag}{L}

\DeclareFontFamily{U}{mathx}{\hyphenchar\font45}
\DeclareFontShape{U}{mathx}{m}{n}{
      <5> <6> <7> <8> <9> <10>
      <10.95> <12> <14.4> <17.28> <20.74> <24.88>
      mathx10
      }{}
\DeclareSymbolFont{mathx}{U}{mathx}{m}{n}
\DeclareFontSubstitution{U}{mathx}{m}{n}
\DeclareMathAccent{\widecheck}{0}{mathx}{"71}

\newtheorem{theorem}{Theorem}[section]

\newtheorem{lemma}[theorem]{Lemma}

\newtheorem{conjecture}[theorem]{Conjecture}
\newtheorem{corollary}[theorem]{Corollary}
\newtheorem{proposition}[theorem]{Proposition}

\theoremstyle{definition}
\newtheorem{definition}[theorem]{Definition}
\newtheorem{remark}[theorem]{Remark}

\newtheorem{example}[theorem]{Example}

\makeatletter
\let\@@pmod\pmod
\DeclareRobustCommand{\pmod}{\@ifstar\@pmods\@@pmod}
\def\@pmods#1{\mkern4mu({\operator@font mod}\mkern 6mu#1)}
\makeatother

\makeatletter
\newtheorem*{rep@thm}{\rep@title}
\newcommand{\newreptheorem}[2]{%
\newenvironment{rep#1}[1][0,0]{%
\def\rep@title{#2##1}%
\begin{rep@thm}}%
{\end{rep@thm}}}
\makeatother
\newreptheorem{theorem}{}

\begin{document}

\begin{abstract}
We investigate the question of the existence of a Lagrangian
concordance between two Legendrian knots in $\R^3$. In particular, we
give obstructions to a concordance from an arbitrary knot to the
standard Legendrian unknot, in terms of normal rulings. We also place
strong restrictions on knots that have concordances both to and from
the unknot and construct an infinite family of knots with
non-reversible concordances from the unknot. Finally, we use our
obstructions to present a complete list of knots with up to $14$
crossings that have Legendrian representatives that are Lagrangian slice.
\end{abstract}

\maketitle

%%%%%%%%%%%%%%%%%%%%%%%%%

%!TEX root = concordances.tex

\section{Introduction}
\label{sec:intro}

In symplectic and contact topology, there has been a great deal of
recent interest in the subject of Lagrangian
cobordisms between Legendrian submanifolds; see for example \cite{baldwin-sivek,bty,bst,cghs,chantraine-concordance,
chantraine-symmetric,ehk,ekholmSFT1,
golovko,hayden-sabloff,
sabloff-traynor-leg}. A key motivation is that one can construct a
category whose objects are
Legendrian submanifolds and whose morphisms are exact Lagrangian
cobordisms, and this category fits nicely into Symplectic Field Theory
\cite{egh}. In particular, Legendrian contact homology gives a functor
from the category of Legendrians to the category of differential
graded algebras.

In this paper, we restrict ourselves to the setting of Legendrian
knots in standard contact $\R^3$, and address the question of when
there exists an exact Lagrangian concordance between two such knots.
Let $\R^3$ be equipped with the standard contact structure $\ker
\alpha$ with $\alpha = dz-y\,dx$, and let $\R^4 = \R_t \times \R^3$ be
the symplectization of $\R^3$, with symplectic form
$\omega = d(e^t\alpha)$. Recall that a knot $\Lambda\subset\R^3$ is
\textit{Legendrian} if $\alpha|_\Lambda = 0$, and a surface $L
\subset \R^4$ is \textit{Lagrangian} if $\omega|_L = 0$.

\begin{definition}
Let $\Lambda_-,\Lambda_+\subset\R^3$ be Legendrian knots. A
\textit{Lagrangian cobordism from $\Lambda_-$ to $\Lambda_+$} is an embedded
Lagrangian $L \subset \R^4$ such that
\begin{align*}
L \cap ((-\infty,-T] \times \R^3) &= (-\infty,-T] \times \Lambda_- 
\\
L \cap ([T,\infty) \times \R^3) &= [T,\infty) \times \Lambda_+ 
\end{align*}
for some $T>0$. This cobordism is \textit{exact} if there exists
$f :\thinspace L
\to \R$ such that $df = \alpha|_L$.
A Lagrangian cobordism of genus $0$ (i.e., a cylinder) is a
\textit{Lagrangian concordance}. Define a relation $\prec$ on the set
of Legendrian knots by $\Lambda_- \prec \Lambda_+$ if there is a
Lagrangian concordance from $\Lambda_-$ to $\Lambda_+$.
\end{definition}

\noindent
We note that any Lagrangian concordance is automatically exact: since
$\alpha$ vanishes on $\Lambda_-$ and $\Lambda_-$ generates $H_1(L)$ if
$L$ is a concordance,
$\alpha$ must equal $0$ in the de Rham cohomology of $L$.

It is clear that $\prec$ is transitive. If $\Lambda_0,\Lambda_1$ are
isotopic as Legendrian knots, then $\Lambda_0 \prec \Lambda_1$ and
$\Lambda_1 \prec \Lambda_0$ \cite{chantraine-concordance}. It follows
that $\prec$ descends to a well-defined, reflexive relation on the set of
isotopy classes of Legendrian knots. In \cite{chantraine-symmetric} it
is shown that $\prec$ is not symmetric (see below for further discussion).

At present, it is unknown whether $\prec$ is antisymmetric: that is,
if $\Lambda_0 \prec \Lambda_1$ and $\Lambda_1 \prec \Lambda_0$, must
$\Lambda_0$ be Legendrian isotopic to $\Lambda_1$? We remark that although our
  definition of $\prec$ involves only concordances rather than general
  cobordisms, this is no restriction in this setting: by a result of
  Chantraine \cite{chantraine-concordance}, if there is a Lagrangian
  cobordism $L$ from $\Lambda_0$ to $\Lambda_1$, then $\tb(\Lambda_1)
  - \tb(\Lambda_0) = 2g(L) \geq 0$, where $\tb$ is the
  Thurston--Bennequin number, and so the existence of cobordisms
  in both directions between $\Lambda_0$ and $\Lambda_1$ implies that
  $\tb(\Lambda_0) = \tb(\Lambda_1)$ and the cobordisms are concordances.

The special case when one of the Legendrian knots is the standard
$\tb=-1$ unknot $U$ is of particular interest. A cobordism from $U$ to
a Legendrian knot $\Lambda$ can be filled at the negative end by a
Lagrangian disk, resulting in a \textit{Lagrangian filling} of
$\Lambda$; such fillings are relatively common (see
e.g. \cite{hayden-sabloff}). In the case when the cobordism is a
cylinder ($U \prec \Lambda$), the smooth knot type of $\Lambda$ must
be smoothly slice, and we say that $\Lambda$ is \textit{Lagrangian
  slice}. It is currently unknown whether any Lagrangian slice knot
besides $U$ is concordant \textit{to} $U$.
Chantraine proved in \cite{chantraine-symmetric}, using the
augmentation category of Bourgeois--Chantraine \cite{bc}, that there
is a Lagrangian slice knot $\Lambda$ of type $\overline{9_{46}}$
(the mirror of $9_{46}$) such that $\Lambda \not\prec U$.

One of the goals of this paper is to give strong and easily computable
obstructions to the
existence of a concordance $\Lambda \prec U$. In particular, we show
the following result.

\begin{theorem}[see Theorem~\ref{thm:two-rulings}]
If $\Lambda$ has at least two normal rulings,
\label{thm:two-rulings-intro}
then $\Lambda\not\prec U$.
\end{theorem}

\noindent
In particular, the $\overline{9_{46}}$ knot considered by Chantraine satisfies
this condition,
and so we have a new, simpler proof of Chantraine's result. We also
use this result in Theorem~\ref{thm:infinitely-many-irreversible}
to construct an infinite family of Legendrian
knots $\Lambda$ with $U \prec \Lambda$ and $\Lambda \not\prec U$.
It should be noted that Baldwin and 
the third author \cite{baldwin-sivek}
previously presented a different infinite family of non-reversible
concordances, involving stabilized unknots rather than $U$.

The proof of Theorem~\ref{thm:two-rulings-intro} involves two
ingredients. One is the fact 
that exact
Lagrangian cobordisms induce maps on Legendrian contact homology
\cite{ehk}, and in particular that an exact Lagrangian filling of $\Lambda_-$
induces an augmentation for the differential graded algebra of $\Lambda_+$.
The second is a
study of a particular $2$-cable of Legendrian knots to prove
Theorem~\ref{thm:two-rulings-intro}.  This study relies in turn on the following
observation; see Section~\ref{ssec:solid-torus-knots} for the definition of
Legendrian satellite.

\begin{theorem}[see Theorem~\ref{thm:satellite}]
If $\Lambda_-, \Lambda_+$ are Legendrian knots in $\R^3$ such that
$\Lambda_- \prec \Lambda_+$, then their Legendrian satellites satisfy
$S(\Lambda_-,\Lambda') \prec S(\Lambda_+,\Lambda')$ for any
Legendrian solid torus knot $\Lambda' \subset J^1(S^1)$.
\end{theorem}

\noindent Applying this result not just to $\Lambda$ but to a particular family of
satellites of $\Lambda$, we arrive at an infinite family of obstructions to $\Lambda \prec U$
which depend only on the underlying smooth knot type.

\begin{theorem}[see Theorem~\ref{thm:homfly-kauffman-obstruction}]
Let $\Lambda$ be a Legendrian knot of smooth knot type $K$, and let $K_n$
denote the 0-framed $n$-cable of $K$.  If $U \prec \Lambda \prec U$, and
$p_n(a,z)$ is either the HOMFLY-PT or Kauffman (Dubrovnik)
polynomial of
$K_n$, then $\maxdeg_a p_n(a,z) = n-1$ and $p_n(a,z)$ has $a^{n-1}$-coefficient
equal to $z^{1-n}$ for all $n \geq 1$.
\end{theorem}

\noindent
One can in particular use this result to obstruct $\Lambda\prec U$
when $U\prec\Lambda$, by
finding two distinct rulings of a Legendrian representative
$S(\Lambda,tw_n)$ of $K_n$; one ruling is guaranteed by the existence
of a ruling on $\Lambda$, but the second violates the conclusion of
Theorem~\ref{thm:homfly-kauffman-obstruction}. See
Theorem~\ref{thm:twn-ruling-polynomials} for more details and
Section~\ref{ssec:nonrev} for applications of this technique.

In a different direction, a result of Ekholm, Honda, and K\'alm\'an
\cite{ehk} gives 
explicit exact Lagrangian cobordisms between Legendrian knots whose
fronts are related by two elementary moves, unknot filling and pinch
moves, which correspond topologically to $0$-handle and $1$-handle
attachment. We call a Lagrangian cobordism \textit{decomposable} if it
can be broken into these elementary pieces; decomposable cobordisms
currently form a central tool for constructing exact Lagrangian
cobordisms.

Although we do not answer the general question of whether
a nontrivial Legendrian knot $\Lambda$ can satisfy $U\prec \Lambda\prec U$, we prove a special case of this in
Section~\ref{sec:unknots}: no nontrivial $\Lambda$ can have a
decomposable Lagrangian concordance to $U$ (see Theorem~\ref{thm:unknot}). The proof of this is
purely topological, relying on work of Kronheimer and Mrowka
\cite{km-excision}.  We do however expect that indecomposable Lagrangian
concordances exist, and we exhibit a possible example in
Conjecture~\ref{conj:indecomposable}, but we do not know of any
potential examples from $U$ to another knot.

Finally, in Section~\ref{sec:census}, we enumerate all knots up
through $14$ crossings with Legendrian representatives that are
Lagrangian slice. Necessary conditions for a knot $K$ to have such a
Legendrian representative are that $K$ must be smoothly slice and
satisfy $\maxtb(K) = -1$, where $\maxtb$ is maximal
Thurston--Bennequin number. Through $14$ crossings, we show that these
conditions are sufficient as well, using explicit decomposable
cobordisms for each knot type. To
help with the census of Lagrangian slice knots, we prove that no
nontrivial alternating knot can be Lagrangian slice
(see Theorem~\ref{thm:alternating-positive}), which has the
side benefit of giving a new, contact-geometric proof of a result of
Nakamura \cite{nakamura} that any reduced alternating diagram of a
positive knot can only have positive crossings.

\subsection*{Acknowledgments}
CC was supported by a CIRGET postdoctoral fellowship. 
LN thanks Tobias Ekholm for many useful conversations. LN was
supported by NSF grants DMS-0846346 and DMS-1406371. SS was supported
by NSF postdoctoral fellowship DMS-1204387. 
%!TEX root = concordances.tex

\section{Legendrian Satellites and Concordance}
\label{sec:satellites}

In this section, we present new obstructions to the existence of a
Lagrangian concordance between two Legendrian knots. We assume basic
familiarity with the theory of Legendrian knots, along the lines of
\cite{etnyre-survey}. Throughout this section (and indeed, for the rest of the paper), we use $U$ to denote the standard Legendrian unknot with $\tb=-1$.

\subsection{Review of functoriality of Legendrian contact homology under cobordisms}

Associated to any Legendrian knot in $\R^3$ is the Chekanov--Eliashberg differential graded algebra (DGA) \cite{chekanov,eliashberg-icm}, whose homology is the Legendrian contact homology of the knot. As part of the Symplectic Field Theory package \cite{egh}, this DGA behaves functorially under exact Lagrangian cobordism. Since this behavior underlies our study of obstructions, we briefly review the statement here, as proved by Ekholm, Honda, and K\'alm\'an \cite{ehk}.

\begin{proposition}[\cite{ehk}]
\label{prop:ehk}
If $\Lambda_-,\Lambda_+$ are Legendrian knots such that there is an exact Lagrangian cobordism from $\Lambda_-$ to $\Lambda_+$, then there is a morphism of Legendrian contact homology DGAs
\[
\mathcal{A}_{\Lambda_+} \to \mathcal{A}_{\Lambda_-}.
\]
\end{proposition}

For our purposes, it will be convenient to clarify the statement of Proposition~\ref{prop:ehk} in two ways, which we present as the following two remarks.

\begin{remark}[Coefficient ring]
The morphism in
Proposition~\ref{prop:ehk} restricts to the identity map on the
coefficient ring of the DGA, which in \cite{ehk} is $\zt =
\zz/2\zz$. In fact, Proposition~\ref{prop:ehk} can be extended to give
a morphism of DGAs over $\zz[t,t^{-1}]$, where the coefficients are
lifted to $\zz[t,t^{-1}]$ as in \cite{etnyre-ng-sabloff}. As noted in
\cite{ehk}, a proof over $\zz$ (or $\zz[t,t^{-1}]$) would entail a
consideration of orientations of moduli spaces. However, working mod
$2$ and lifting Proposition~\ref{prop:ehk} to DGAs over
$\zt[t,t^{-1}]$ simply entails choosing base points on both ends of
the concordance, joining the base points by a path on the concordance
cylinder, and keeping track of intersections of boundaries of
holomorphic disks with this path. We omit the details of the proof
here.
\end{remark}

\begin{remark}[Grading]
The extent to which the morphism $\phi$ in Proposition~\ref{prop:ehk}
preserves the grading in the DGAs depends on the Maslov index of the
Lagrangian cobordism $\Lag$, defined to be the $\gcd$ of the Maslov indices of
all closed curves in $\Lag$, including $\Lambda_-$ and $\Lambda_+$. If $\Lag$ is
oriented, then $\phi$ preserves grading mod $2$; if $\Lag$ is
unoriented, then $\phi$ may not preserve the grading at
all. There is no reason in general that $\phi$ needs to preserve
the full $\zz$ grading, even if $\Lambda_-$ and $\Lambda_+$ have rotation number
$0$. However, we will be particularly interested below in the special case where $\Lambda_-$ is the standard
Legendrian unknot $U$. In this case,
since $U$ can be filled in with a Lagrangian disk $D$ and curves in
$\Lag$ are null-homotopic in $\Lag \cup D$, $\phi$ does preserve the full
$\zz$ grading.
\end{remark}

\subsection{Solid torus knots and concordance}
\label{ssec:solid-torus-knots}

Our obstructions rely on considering satellites of Legendrian knots.
We begin by reviewing Legendrian solid torus knots and the Legendrian satellite construction from %Ng--Traynor
\cite{ng-traynor}.  We identify the open solid torus $S^1 \times \R^2$ as the 1-jet space of the circle, i.e.\ as $J^1(S^1) \cong T^*S^1 \times \R$, which equips it with a natural contact form $\alpha = dz - y\,dx$; here $x$ and $y$ are the base and fiber coordinates on $T^*S^1$ and $z$ is the $\R$-coordinate.  Just as in the case of $\R^3 \cong J^1(\R)$, we can recover a Legendrian knot from its front projection onto $S^1_x \times \R_z$, which in practice is drawn by representing $S^1$ as an interval and identifying its endpoints.

Given a Legendrian companion knot $\Lambda \subset \R^3$ and a Legendrian pattern knot $\Lambda'\subset J^1(S^1)$, the contact neighborhood theorem says that $\Lambda$ has a standard neighborhood $N(\Lambda)$ for which there is a contactomorphism $\varphi: J^1(S^1)\isomto N(\Lambda)$, and we define the \emph{Legendrian satellite} $S(\Lambda,\Lambda') \subset \R^3$ to be the image $\varphi(\Lambda')$.  We remark that this requires a choice of framing for $N(\Lambda)$, which we fix to be the contact (Thurston--Bennequin) framing.

We can produce a front projection of $S(\Lambda,\Lambda')$ as follows.  If the front projection of $L$ intersects the ends of the $S^1$ interval in $n$ points, then we produce a front for the \emph{$n$-copy} of $\Lambda$ by taking $n$ copies of the front for $\Lambda$ and shifting each one a very small distance in the $z$-direction.  (Topologically, the $n$-copy is the $(n,n\cdot \tb(\Lambda))$-cable of $\Lambda$, in which the first coordinate denotes the longitudinal winding.)  We then take a point where $\Lambda$ is oriented from left to right, cut the front open along the $n$-copy of that point, and insert the front diagram for $\Lambda'$.  See Figure~\ref{fig:satellite-example} for an example satellite in which the companion $\Lambda$ is a right handed trefoil with $(\tb,r)=(1,0)$ and the pattern $\Delta_2$ is a positive half twist on two strands, whose name is borrowed from \cite{ng-rutherford}.  Topologically, $S(\Lambda,\Delta_2)$ is the $(2,3)$-cable of the right handed trefoil.

\begin{figure}[t]
\labellist
\huge
\pinlabel $\leadsto$ at 145 44
\endlabellist
\begin{centering}
\includegraphics{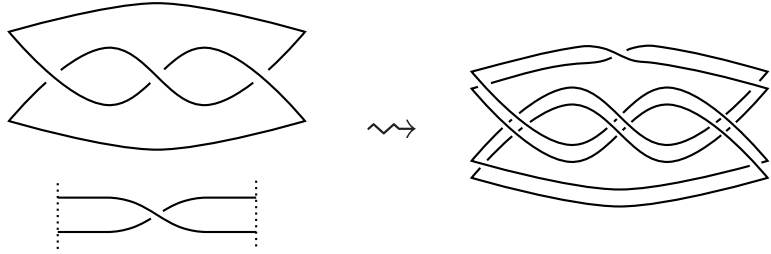}
\end{centering}
\caption{Using the companion $\Lambda$ and pattern $\Delta_2$ to produce the satellite $S(\Lambda,\Delta_2)$.}
\label{fig:satellite-example}
\end{figure}

Our main theorem says that Lagrangian concordance is preserved by Legendrian satellite operations.  This is an analogue of a well-known fact in classical knot concordance, and indeed the proofs are nearly identical once we observe that Lagrangian cylinders have standard neighborhoods.

\begin{theorem}
\label{thm:satellite}
Let $\Lambda'$ be a Legendrian knot in $J^1(S^1)$.  If $\Lambda_-$ and $\Lambda_+$ are Legendrian knots in $\R^3$ such that $\Lambda_- \prec \Lambda_+$, then $S(\Lambda_-,\Lambda') \prec S(\Lambda_+,\Lambda')$.
\end{theorem}

\begin{proof}
Let $\Lag$ be a Lagrangian concordance from $\Lambda_-$ to $\Lambda_+$ in the symplectization 
$\R \times \R^3$, and restrict to some 
$[-T,T] \times \R^3$ such that $\Lag$ is a product cylinder outside this region.  If $\Lambda_-' \subset J^1(S^1)$ is the core of the solid torus given by $y=z=0$, then $\Lambda_-'$ is Legendrian and so the cylinder 
$[-T,T] \times \Lambda_-' \subset [-T,T] \times J^1(S^1)$
is also Lagrangian.  The Weinstein 
neighborhood theorem thus provides a symplectomorphism 
$\varphi: N([-T,T]\times \Lambda_-') \isomto N(\Lag)$ between neighborhoods of the two Lagrangian cylinders, and since $[-T,T]$ is compact we can isotop $\Lambda'$ close enough to $\Lambda_-'$ to ensure that the Lagrangian 
$[-T,T]\times \Lambda'$ lies inside $N([-T,T]\times \Lambda_-')$.

Choosing the Thurston--Bennequin framing on $\Lambda_-$ for the neighborhood of $\{-T\}\times \Lambda_-$ in $\{-T\}\times\R^3$, it follows that $\varphi(\{-T\}\times \Lambda')$ is the Legendrian satellite $S(\Lambda_0,\Lambda')$ and that $\varphi([-T,T]\times \Lambda')$ is a Lagrangian cylinder.  Its restriction to $\{T\}\times\R^3$ is a Legendrian satellite of $\Lambda_1$, and since $\Lambda_- \prec \Lambda_+$ implies that $\tb(\Lambda_-)=\tb(\Lambda_+)$ by \cite{chantraine-concordance}, this satellite is also $\tb$-framed, hence it is $S(\Lambda_1,\Lambda')$.  We conclude after gluing on cylindrical ends that $\varphi([-T,T]\times \Lambda')$ is the desired concordance.
\end{proof}

\subsection{$A_{(2)}$-compatibility and the Legendrian unknot}
\label{ssec:a2-compatibility}

In this section, we examine a particular obstruction to the existence
of Lagrangian concordances.
Following \cite{ng-rutherford}, we say that a Legendrian knot $\Lambda$ is \emph{$A_{(2)}$-compatible} if the satellite $S(\Lambda,\Delta_2)$ admits a normal ruling, or equivalently if it admits an augmentation, i.e.\ a DGA morphism $\mathcal{A}_{S(\Lambda,\Delta_2)} \to \zt= \zz/2\zz$  \cite{fuchs, fuchs-ishkhanov, sabloff-ruling}.  The standard Legendrian unknot $U$
is not $A_{(2)}$-compatible, since the satellite $S(U,\Delta_2)$ is a topological unknot with $\tb=-3$ and hence a stabilization.

\begin{theorem}
\label{thm:a2-compatible}
Let $\Lambda$ be a Legendrian knot.  If $\Lambda \prec U$ then $\Lambda$ is not $A_{(2)}$-compatible.
\end{theorem}

\begin{proof}
If $\Lambda \prec U$ then Theorem~\ref{thm:satellite} says that $S(\Lambda,\Delta_2) \prec S(U,\Delta_2)$, and then Proposition~\ref{prop:ehk} provides a morphism $\mathcal{A}_{S(U,\Delta_2)} \to \mathcal{A}_{S(\Lambda,\Delta_2)}$ between the Legendrian contact homology DGAs of the two satellites over $\zt = \zz/2\zz$. If $\Lambda$ were $A_{(2)}$-compatible then composing this morphism with an augmentation $\mathcal{A}_{S(\Lambda,\Delta_2)} \to \zt$ would give an augmentation of $S(U,\Delta_2)$, which does not exist.
\end{proof}

\begin{remark}
\label{rmk:946}
We know from \cite{ng-rutherford} that $\Lambda$ is $A_{(2)}$-compatible if and only if the DGA $\mathcal{A}_\Lambda$ defined over $\zt[t,t^{-1}]$ has a $2$-dimensional representation sending $t$ to $\left(\begin{smallmatrix} p & 1 \\ 1 & 0\end{smallmatrix}\right)$ for some $p$.  Since $\mathcal{A}_U$ has no such representations, Theorem~\ref{thm:a2-compatible}
also follows from the extension of Proposition~\ref{prop:ehk} to
$\zt[t,t^{-1}]$ coefficients. This can be reworked to give an
obstruction to concordance to $U$ that does not explicitly mention
$A_{(2)}$-compatibility. For example, for the $\overline{9_{46}}$ knot
$\Lambda$ considered by Chantraine \cite{chantraine-symmetric} (see
Example~\ref{ex:ruling-2copy-946} below), $\Lambda \not\prec U$ since it can
be checked that the DGA for $\Lambda$ over $\zt[t,t^{-1}]$ has a
$2$-dimensional representation sending $t$ to
$\left(\begin{smallmatrix} 0 & 1 \\ 1 & 0\end{smallmatrix}\right)$
(and in fact another sending $t$ to
$\left(\begin{smallmatrix} 1 & 1 \\ 1 & 0\end{smallmatrix}\right)$), while the DGA for $U$ does not.
\end{remark}

To apply Theorem~\ref{thm:a2-compatible}, it is convenient to have sufficient conditions for $A_{(2)}$-compatibility.
The second author and Rutherford \cite[Theorem~5.4]{ng-rutherford} showed that if the ruling polynomial $R_\Lambda(z)$ has positive degree then $\Lambda$ is $A_{(2)}$-compatible.\footnote{The convention in \cite{ng-rutherford} is that a ruling $\rho$ with $s$ switches and $c$ right cusps contributes $z^{s-c}$ to the ruling polynomial, so their condition is $\deg R_\Lambda(z) \geq 0$.  We instead use the convention $z^{s-c+1}$ of \cite{rutherford}, so that the 2-graded and ungraded ruling polynomials match the appropriate coefficients of the HOMFLY-PT and Kauffman polynomials respectively.}
The following result is similar but allows for $R_\Lambda(z)$ to be a constant as well.

\begin{theorem}
\label{thm:two-rulings}
If $\Lambda$ has at least two normal rulings, then $\Lambda$ is $A_{(2)}$-compatible, and thus $\Lambda \not\prec U$.
\end{theorem}

\begin{proof}
We will use two distinct rulings $\rho_1$ and $\rho_2$ of $\Lambda$ to produce an explicit ruling of $S(\Lambda,\Delta_2)$.  We let $c$ be the rightmost crossing of a front diagram for $\Lambda$ where $\rho_1$ and $\rho_2$ differ, and we label the rulings so that $\rho_1$ does not have a switch at $c$ but $\rho_2$ does.  The $2$-copy of the given front has four crossings for every crossing in the front of $\Lambda$; we place the half-twist $\Delta_2$ inside the four crossings corresponding to $c$ along the $2$-copy of the undercrossing strand.  At all other crossings of $\Lambda$ we place a switch at the corresponding north crossing of the 2-copy if $\rho_1$ has a switch, and likewise for the south crossing if $\rho_2$ has a switch.

At the 2-copy of the distinguished crossing $c$, we place switches at the south and west of these four crossings.  This is illustrated in Figure~\ref{fig:ruling-2copy-crossing} in three different cases.  Since $\rho_1$ and $\rho_2$ are identical to the right of $c$, the ruling we have constructed looks like a ``2-copy'' of either ruling $\rho_i$ to the right of the 2-copy of $c$.  Since $\rho_2$ has a normal switch at $c$, there are three possible ways in which the companions of the strands through $c$ could be positioned relative to each other and to $c$.  Thus in Figure~\ref{fig:ruling-2copy-crossing} we verify that in each case the specified switches are normal, and that to the left of the 2-copy of $c$ the strands in the top and bottom halves of the 2-copy are paired according to $\rho_1$ and $\rho_2$ respectively.

\begin{figure}[t]
\labellist
\tiny \hair 2pt
\pinlabel $1a$ [r] at 0 112
\pinlabel $2a$ [r] at 0 104
\pinlabel $1b$ [r] at 0 96
\pinlabel $2a$ [r] at 0 88
\pinlabel $1a$ [r] at 0 41
\pinlabel $2b$ [r] at 0 33
\pinlabel $1b$ [r] at 0 25
\pinlabel $2b$ [r] at 0 17
\pinlabel $1a$ [r] at 96 96
\pinlabel $2a$ [r] at 96 88
\pinlabel $1b$ [r] at 96 41
\pinlabel $2b$ [r] at 96 33
\pinlabel $1a$ [r] at 96 25
\pinlabel $2b$ [r] at 96 17
\pinlabel $1b$ [r] at 96 9
\pinlabel $2a$ [r] at 96 1
\pinlabel $1a$ [r] at 191 128
\pinlabel $2a$ [r] at 191 120
\pinlabel $1b$ [r] at 191 112
\pinlabel $2b$ [r] at 191 104
\pinlabel $1a$ [r] at 191 96
\pinlabel $2b$ [r] at 191 88
\pinlabel $1b$ [r] at 191 41
\pinlabel $2a$ [r] at 191 33
\endlabellist
\begin{centering}
\includegraphics{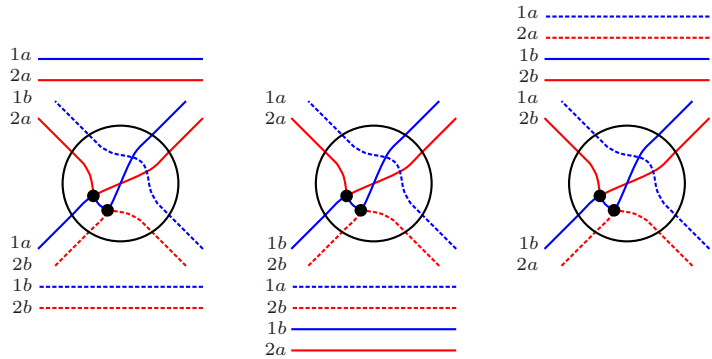}
\end{centering}
\caption{The switches placed at the 2-copy of $c$ in the proof of Theorem~\ref{thm:two-rulings}.  Strands with the same number and letter are paired, and the number indicates whether a strand corresponds to $\rho_1$ or to $\rho_2$.}
\label{fig:ruling-2copy-crossing}
\end{figure}
It is now easy to verify that the set of switches we have described provides a normal ruling of $S(\Lambda,\Delta_2)$, as desired.
\end{proof}

\begin{example}
\label{ex:ruling-2copy-946}
In Figure~\ref{fig:ruling-2copy-946} we display two different rulings
of a Legendrian knot $\Lambda$ of topological type $\overline{9_{46}}$
and the corresponding ruling of $S(\Lambda,\Delta_2)$.  (Here
$\overline{9_{46}}$ denotes the mirror of $9_{46}$, cf.\
Remark~\ref{rmk:chirality}.)
By Theorem~\ref{thm:two-rulings} it follows that there is no Lagrangian concordance from $\Lambda$ to $U$, reproving the main result of \cite{chantraine-symmetric}.
\begin{figure}[t]
\begin{centering}
\includegraphics{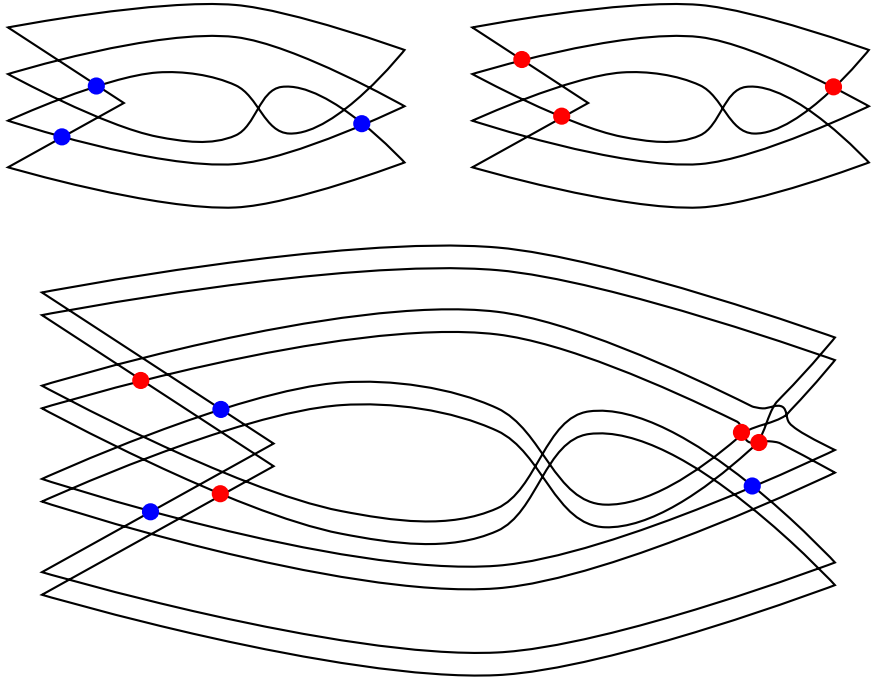}
\end{centering}
\caption{Constructing a normal ruling of $S(\Lambda,\Delta_2)$ from two different rulings of $\Lambda$.}
\label{fig:ruling-2copy-946}
\end{figure}
\end{example}

\begin{corollary}
\label{cor:ruling-polynomial-uku}
If $\Lambda \prec U$ then $\Lambda$ has ungraded ruling polynomial $R_\Lambda(z)$ equal to either $0$ or $1$.  In particular, if $U\prec \Lambda$ as well, then the $d$-graded ruling polynomial of $\Lambda$ is equal to $1$ for all $d$.
\end{corollary}

\begin{proof}
Theorem~\ref{thm:a2-compatible} says that $\Lambda$ is not $A_{(2)}$-compatible, so by \cite[Theorem~5.4]{ng-rutherford} the ungraded ruling polynomial $R_\Lambda(z)$ must have degree 0.  If its constant term is at least 2 then it has two or more rulings, which is ruled out by Theorem~\ref{thm:two-rulings}, so the only remaining possibilities are $0$ and $1$.

If we also have $U\prec \Lambda$, then the morphism $\mathcal{A}_\Lambda \to \mathcal{A}_U \cong \zt$ induced by such a concordance is by definition an augmentation of $\Lambda$.  Thus $\Lambda$ has at least one graded normal ruling, and any other $d$-graded ruling would also be a second ungraded ruling of $\Lambda$, which cannot exist.
\end{proof}

We can now generalize the $\overline{9_{46}}$ example of \cite{chantraine-symmetric} and Example~\ref{ex:ruling-2copy-946} by providing an infinite family of irreversible Lagrangian concordances.  The simplest of these is a Legendrian $\overline{14n_{15581}}$ knot, which will appear later in Table~\ref{table:census-diagrams}.

\begin{theorem}
\label{thm:infinitely-many-irreversible}
There are infinitely many Legendrian knots $\Lambda$ such that $U\prec \Lambda$ but $\Lambda\not\prec U$.
\end{theorem}

\begin{proof}
We can form an infinite family $\Lambda_1,\Lambda_2,\dots$ of Legendrian knots by adding half-twists to the $\overline{14n_{15581}}$ knot of Table~\ref{table:census-diagrams}, as shown in Figure~\ref{fig:family-15581}.  Each $\Lambda_n$ is Lagrangian slice: surgering along the dotted line indicated in the figure produces a two-component Legendrian unlink, and we cap off one component to get a concordance from the Legendrian unknot to $\Lambda_n$, so $U \prec \Lambda_n$.  However, each $\Lambda_n$ admits at least two normal rulings, as illustrated in Figure~\ref{fig:family-rulings} for $\Lambda_3$; we generalize to all $\Lambda_n$ by placing a switch at every one of the added half-twists and by using the same set of switches as in either ruling of Figure~\ref{fig:family-rulings} away from the half-twists.  We conclude that $\Lambda_n \not\prec U$ for all $n$.
\begin{figure}[t]
\begin{centering}
\includegraphics{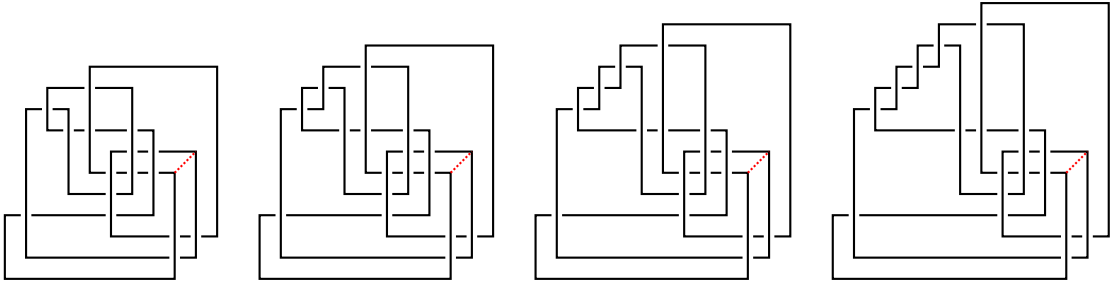}
\end{centering}
\caption{Grid diagrams for the first four members of an infinite family of Lagrangian slice knots, beginning with $\Lambda_1=\overline{14n_{15581}}$.}
\label{fig:family-15581}
\end{figure}
\begin{figure}[t]
\begin{centering}
\includegraphics{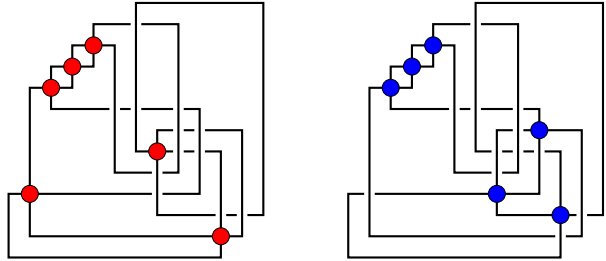}
\end{centering}
\caption{Two rulings of the third member $\Lambda_3$ of the family of Figure \ref{fig:family-15581}.}
\label{fig:family-rulings}
\end{figure}
\end{proof}

\subsection{Satellites which fix the Legendrian unknot $U$}
\label{ssec:fix-unknot}

If $\Lambda'$ is a Legendrian pattern for which $S(U,\Lambda')$ is Legendrian isotopic to $U$, then applying Theorem~\ref{thm:satellite} to any concordance $U \prec \Lambda$ or $\Lambda \prec U$ tells us that $U \prec S(\Lambda,\Lambda')$ or $S(\Lambda,\Lambda') \prec U$ respectively.  In this subsection we will investigate a family of such patterns.

Let $P_n$ denote the knot in $J^1(S^1)$ depicted in Figure~\ref{fig:twist-clasps}, where $n$ is the winding number of $P_n$ around the solid torus.  We can describe $P_n$ as the concatenation of a full positive twist $tw_n$ on $n$ strands and a cascade of $n-1$ clasps, where the $i$th clasp connects the $i$th and $(i+1)$st strands as numbered from top to bottom.

\begin{figure}[t]
\begin{centering}
\includegraphics{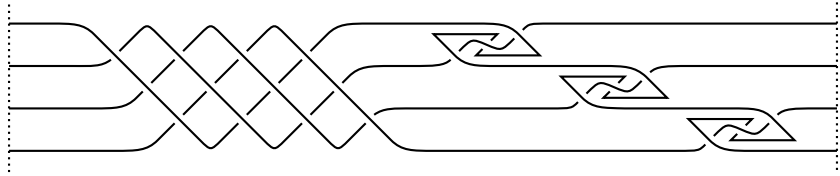}
\end{centering}
\caption{The Legendrian solid torus knot $P_n$, shown here for $n=4$.}
\label{fig:twist-clasps}
\end{figure}

\begin{lemma}
\label{lem:pn-fixes-unknot}
The satellite $S(U,P_n)$ is Legendrian isotopic to $U$ for all $n \geq 1$.
\end{lemma}

\begin{proof}
The lemma is true for $n=1$ by inspection, and a straightforward computation shows that $\tb(S(U,P_n))=-1$, so it suffices to check that $S(U,P_n)$ is topologically isotopic to $S(U,P_{n-1})$ for all $n \geq 2$.  This is illustrated in Figure~\ref{fig:twist-unknot-simplify}: the highlighted portion of $S(U,P_n)$ can be pushed back through the middle of the satellite, lifted behind it, and then twisted to remove the top clasp, and what remains is $S(U,P_{n-1})$.
\begin{figure}[t]
\begin{centering}
\includegraphics{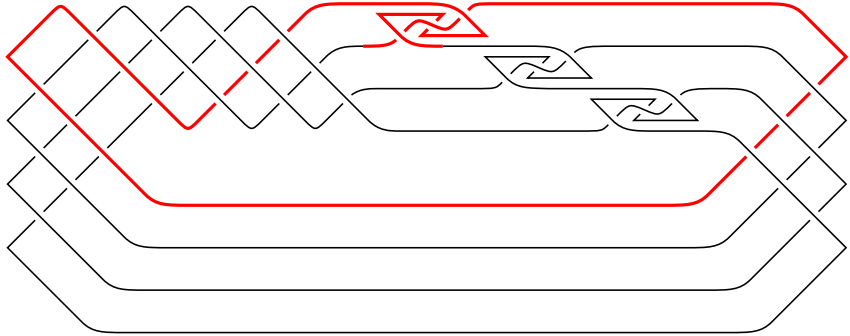}
\end{centering}
\caption{The satellite $S(U,P_4)$ is isotopic to $S(U,P_3)$, as can be seen by moving the highlighted portion appropriately to simplify the knot diagram.}
\label{fig:twist-unknot-simplify}
\end{figure}
\end{proof}

\begin{theorem}
\label{thm:rulings-tw-vs-p}
The $d$-graded ruling polynomials of $S(\Lambda,P_n)$ and $S(\Lambda,tw_n)$ are related by
\[ R^d_{S(\Lambda,P_n)}(z) = z^{n-1} \cdot R^d_{S(\Lambda,tw_n)}(z) \]
for any Legendrian knot $\Lambda$ and any $d \mid 2r(\Lambda)$.
\end{theorem}

\begin{proof}
Each of the $n-1$ clasps in $P_n$ contains three crossings and two right cusps, and the gradings of these crossings are $0\pmod{2r(\Lambda)}$ since the Maslov potentials on parallel strands coming from adjacent copies of $\Lambda$ differ by $2r(\Lambda)$.  It is easy to see that any normal ruling of $S(\Lambda,P_n)$ must have switches at all of these crossings, and that once we remove the pairs of short strands incident to these cusps from the ruling then what remains is a normal ruling of $S(\Lambda,tw_n)$.  Conversely, any ruling of $S(\Lambda,tw_n)$ gives rise to a ruling of $S(\Lambda,P_n)$ in this fashion, and so there is a bijection between the sets of rulings of the two satellites.

Let $\rho_{tw}$ be a ruling of $S(\Lambda,tw_n)$ with $s$ switches, and suppose that $S(\Lambda,tw_n)$ has $c$ right cusps.  Then $\rho_{tw}$ contributes $z^{s-c+1}$ to its ruling polynomial.  The above bijection pairs $\rho_{tw}$ with some ruling $\rho_P$ of $S(\Lambda,P_n)$ which has $s+3(n-1)$ switches, and $S(\Lambda,P_n)$ has $c+2(n-1)$ right cusps, so $\rho_P$ contributes
\[ z^{(s+3(n-1))-(c+2(n-1))+1} = z^{n-1} \cdot z^{s-c+1} \]
to the ruling polynomial of $S(\Lambda,P_n)$.  Since this is true for all rulings $\rho_{tw}$ of $S(\Lambda,P_n)$, the ruling polynomials of the two satellites differ by a factor of $z^{n-1}$ as desired.
\end{proof}

\begin{theorem}
\label{thm:twn-ruling-polynomials}
Let $\Lambda$ be a Legendrian knot such that $U \prec \Lambda \prec U$.  Then $S(\Lambda,tw_n)$ has $d$-graded ruling polynomial $R^d_{S(\Lambda,tw_n)}(z) = z^{1-n}$ for all $n \geq 1$ and all $d$.
\end{theorem}

\begin{proof}
Taking satellites $S(\cdot, P_n)$ and applying Theorem~\ref{thm:satellite} gives $U = S(U,P_n) \prec S(\Lambda,P_n)$ and $S(\Lambda,P_n) \prec S(U,P_n) = U$.  This implies that $r(S(\Lambda,P_n))=r(U)=0$, hence from Corollary~\ref{cor:ruling-polynomial-uku} we must have $R^d_{S(\Lambda,P_n)}(z)=1$ for all $d$ and the conclusion about $R^d_{S(\Lambda,tw_n)}(z)$ is a consequence of Theorem~\ref{thm:rulings-tw-vs-p}.
\end{proof}

\begin{remark}
Under the hypotheses of Theorem~\ref{thm:twn-ruling-polynomials}, the satellite $S(\Lambda,tw_n)$ has exactly one ruling for each $n$.  We can construct such a ruling as the $n$-copy of the unique ruling of $\Lambda$, by taking each switch of that ruling and placing switches among the corresponding $n^2$ crossings of $S(\Lambda,tw_n)$ at exactly the $n$ crossings where both strands belong to the same component.  Thus if $U\prec \Lambda$ then we can prove $\Lambda\not\prec U$ by exhibiting a single ruling of some $S(\Lambda,tw_n)$ (or even of the $n$-copy of $\Lambda$) which is not the $n$-copy of a ruling of $\Lambda$.
\end{remark}

Since a knot $\Lambda$ for which $U \prec \Lambda \prec U$ must have Thurston--Bennequin invariant $-1$, the $n$-copy of $\Lambda$ is topologically the $(n,-n)$-cable of $\Lambda$.  We obtain $S(\Lambda,tw_n)$ from the $n$-copy by inserting a full positive twist, so
the Legendrian satellite $S(\Lambda,tw_n)$ is topologically the $(n,0)$-cable of $\Lambda$, i.e.\ the Seifert-framed $n$-stranded cable of $\Lambda$.  We will denote this $n$-cable by $\Lambda_n$.

In the cases $d=1$ and $d=2$, Rutherford \cite{rutherford} showed that
the $d$-graded ruling polynomials of a Legendrian link
$\Lambda$ of topological type $K$ are determined by its Kauffman polynomial $F_K(a,z)$ (the Dubrovnik version) and its HOMFLY-PT polynomial $P_K(a,z)$.  More precisely, these polynomials both have maximum $a$-degree at most $-\maxtb(K)-1$ \cite{rudolph-kauffman, franks-williams, morton}, and  in fact he proved that $R^1_\Lambda(z)$ and $R^2_\Lambda(z)$ are the coefficients of $a^{-\tb(\Lambda)-1}$ in $F_K(a,z)$ and $P_K(a,z)$, respectively.  We can compute that $\tb(S(\Lambda,tw_n)) = n^2 (\tb(\Lambda)+1) - n$, so if $U \prec \Lambda \prec U$, then $\tb(S(\Lambda,tw_n)) = -n$.  We have therefore proved the following.

\begin{theorem}
\label{thm:homfly-kauffman-obstruction}
If a smooth knot $K$ has a Legendrian representative $\Lambda$ such that $U \prec \Lambda \prec U$, and $K_n$ denotes the topological $n$-cable of $K$, then the HOMFLY-PT and Kauffman polynomials $P_{K_n}(a,z)$ and $F_{K_n}(a,z)$ both have maximal $a$-degree equal to $n-1$ and $a^{n-1}$-coefficient equal to $z^{1-n}$ for all $n \geq 1$.
\end{theorem}

\begin{remark}
The conclusion of
Theorem~\ref{thm:homfly-kauffman-obstruction} depends only on the
smooth knot type, so if the theorem can be used to prove that
$\Lambda\not\prec U$ for one Legendrian representative $\Lambda$ of
$K$ with $U\prec \Lambda$, then it proves $\Lambda\not\prec U$ for all
such representatives.
\label{rmk:smoothtype}
\end{remark} 
%!TEX root = concordances.tex

\section{Concordances to the Unknot}
\label{sec:unknots}

We say that a Lagrangian concordance $\Lambda_0 \prec \Lambda_1$ is \emph{decomposable} if it can be built as a composition of elementary moves, namely isotopies, minimum cobordisms, and saddle cobordisms in the language of \cite[Section 6]{ehk}.  Our goal in this section is to prove that $\Lambda_1$ can never be a topological unknot in such a concordance; the proof will make no use of any contact topology, proceeding instead by considering branched double covers of knots.  We thus begin with the following theorem, which may be of independent interest.

\begin{theorem}
\label{thm:double-cover-representation}
Let $\Sigma(K)$ denote the double cover of $S^3$ branched over the knot $K$.  Then the fundamental group of $\Sigma(K)$ admits a nontrivial representation
\[ \rho: \pi_1(\Sigma(K)) \to SO(3) \]
if and only if $K$ is not an unknot.
\end{theorem}

\begin{proof}
If $K$ is the unknot then this is immediate, since $\Sigma(K) \cong S^3$ is simply connected.

For nontrivial $K$, a theorem of Kronheimer and Mrowka \cite[Corollary 7.17]{km-excision} says that given a meridian $\mu$ of $K$, there is an irreducible homomorphism $\varphi: \pi_1(S^3 \ssm K) \to SU(2)$ such that $\varphi(\mu) = i$.  (Here we view $SU(2)$ as the unit quaternions.)  Since $\pi_1(S^3\ssm K)$ is normally generated by $\mu$, it has a unique normal subgroup $N$ of index 2, namely the kernel of the composition of the abelianization and mod 2 reduction maps $\pi_1(S^3\ssm K) \to \zz \to \zz/2\zz$, corresponding to the double cover of $S^3 \ssm K$.  This subgroup contains $\mu^2$, and in fact $\pi_1(\Sigma(K)) = N/\langle \mu^2\rangle$, so we wish to use $\varphi$ to produce a map $N/\langle \mu^2\rangle \to SO(3)$.

Consider the composition $\tilde{\varphi}: \pi_1(S^3\ssm K)\to SO(3)$ of $\varphi$ with the quotient $SU(2)\to SO(3)$.  Since $\varphi(\mu^2) = i^2 = -1$ in $SU(2)$, it follows that $\tilde{\varphi}(\mu^2) = 1$ in $SO(3)$, and so $\tilde{\varphi}|_N$ descends to a map
\[ \rho: \pi_1(\Sigma(K)) = N/\langle \mu^2\rangle \to SO(3). \]
We need to check that $\rho$ is nontrivial.  But since $\pi_1(S^3\ssm K)$ is generated by meridians of $K$ and $\varphi$ is irreducible, there must be some meridian $\nu$ such that $\varphi(\nu) \neq \pm i$.  The product $\mu\nu$ lies in $N$, and since $\varphi(\nu) \neq \pm i$ we cannot have $\varphi(\mu\nu) = \pm 1$, so $\rho(\mu\nu) \neq 1$ as desired.
\end{proof}

\begin{theorem}
There are no decomposable Lagrangian concordances of the form $\Lambda
\prec \tilde{U}$, where $\Lambda$ is a topologically nontrivial Legendrian
knot and $\tilde{U}$ is any topologically unknotted Legendrian knot.
\label{thm:unknot}
\end{theorem}

\begin{proof}
Given such a concordance $\Lag$, we can take a front diagram for $\tilde{U}$
and perform a sequence of pinch moves, isotopies, and capping off
$\tb=-1$ unknots to produce a front diagram for $K$, where $K$ is the
smooth knot type of $\Lambda$.  Since $\Lag$ has Euler characteristic zero, there must be exactly as many pinch moves as there are capping moves; call this number $n$.  Moreover, we can postpone any capping moves until the end by taking $\tb=-1$ unknots $U$ which are about to be capped off and instead isotoping them far away from the rest of the diagram so that they no longer interact with any other components.  Thus we can suppose that there is a sequence of Legendrian isotopies and $n$ pinch moves which transforms a front diagram for $\tilde{U}$ into a front diagram for $\Lambda \sqcup nU$.

A pinch move can be performed taking a small ball $B^3$ with two unknotted arcs passing through it, and replacing those arcs with a different pair of unknotted arcs.  The branched double cover of $B^3$ over either pair of arcs is a solid torus, so pinch moves correspond to Dehn surgeries on the branched double cover.  It follows that there must be an $n$-component link $L = L_1 \cup \dots \cup L_n$ in $\Sigma(\tilde{U}) = S^3$ upon which some Dehn surgery produces $\Sigma(K \sqcup nU) = \Sigma(K) \# n(S^1\times S^2)$.  Since $H_1(\Sigma(K) \# n(S^1\times S^2)) = H_1(\Sigma(K)) \oplus \zz^n$ is presented by the $n\times n$ framing matrix of $L$, we must have $H_1(\Sigma(K)) = 0$ and the framing matrix must be identically zero.  In particular, we must perform $0$-surgery on each $L_i$.

We can now build a $4$-dimensional handlebody $X$ with boundary $\Sigma(K)$ by taking $B^4$, attaching $n$ $0$-framed $2$-handles to its boundary along $L$ to produce $\Sigma(K) \# n(S^1\times S^2)$ on the boundary, and then attaching $n$ $3$-handles to eliminate each of the $S^1 \times S^2$ summands.  Then $X$ does not have any $1$-handles, and an easy exercise shows that $X$ is contractible since $H_1(\Sigma(K))=0$.  Following an argument from \cite[Section I.3]{kirby-book}, we now turn $X$ upside down to construct it by attaching $n$ 1-handles, $n$ 2-handles, and a $4$-handle to $\Sigma(K)$, and thus we see that the trivial group can be constructed from $\pi_1(\Sigma(K))$ by adding $n$ generators and $n$ relations.  But this is ruled out by a theorem of Gerstenhaber and Rothaus \cite[Theorem 3]{gerstenhaber-rothaus}, since $\pi_1(\Sigma(K))$ admits a nontrivial map into the compact connected Lie group $SO(3)$ by Theorem~\ref{thm:double-cover-representation}, so the claimed decomposable concordance cannot exist.
\end{proof}

Since Legendrian representatives of the unknot are uniquely determined by their classical invariants $tb$ and $r$ \cite{eliashberg-fraser}, it follows that there can only be a decomposable Lagrangian concordance $\Lambda \prec \tilde{U}$ if $\Lambda$ is Legendrian isotopic to $\tilde{U}$.

One might be tempted to conjecture that all Lagrangian concordances
are decomposable, but we believe that this is not the case. If
$W \subset J^1(S^1)$ denotes the Legendrian Whitehead double pattern
$W_0$ of \cite[Figure~23]{ng-traynor}
and $U$ is the standard Legendrian unknot as usual, then
$S(U,W)$ is the Legendrian
right-handed trefoil $T$ with $tb=1$; see Figure~\ref{fig:wh-unknot}.
\begin{figure}[t]
\labellist
\large
\pinlabel $\leadsto$ at 86 15
\pinlabel $\leadsto$ at 185 15
\pinlabel $\leadsto$ at 285 15
\endlabellist
\begin{centering}
\includegraphics{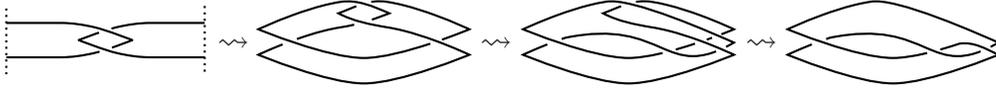}
\end{centering}
\caption{The satellite $S(U,W)$ is a right-handed trefoil with $tb=1$.}
\label{fig:wh-unknot}
\end{figure}
Letting $\Lambda$ be the Legendrian $\overline{9_{46}}$ knot of Example~\ref{ex:ruling-2copy-946}, for which $U \prec \Lambda$, we apply Theorem~\ref{thm:satellite} to construct a Lagrangian concordance $C$ from $T$ to $S(\Lambda,W)$; these knots are displayed in Figure~\ref{fig:indecomposable}.  Topologically, $S(\Lambda,W)$ is the positively clasped, $-1$-twisted Whitehead double of $\overline{9_{46}}$.
\begin{figure}[t]
\labellist
\large
\pinlabel $\prec$ at 128 62
\endlabellist
\begin{centering}
\includegraphics{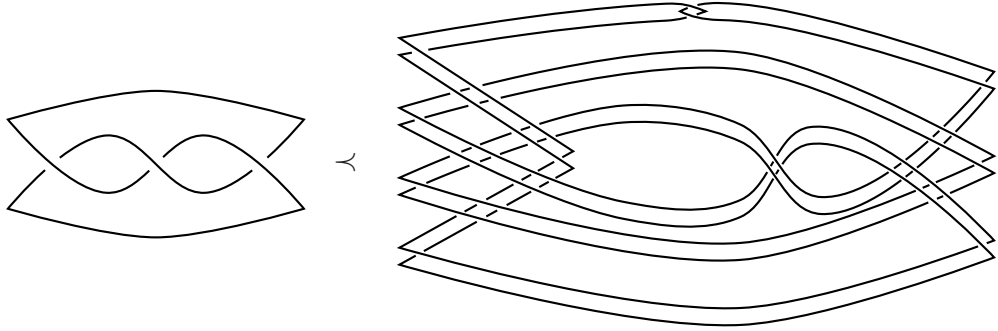}
\end{centering}
\caption{There is a Lagrangian concordance from $T$ to $S(\Lambda,W)$.}
\label{fig:indecomposable}
\end{figure}

\begin{conjecture}
\label{conj:indecomposable}
The Lagrangian concordance $C$ from $T$ to $S(\Lambda,W)$ is not decomposable.
\end{conjecture}
\noindent We believe that this conjecture may hold more generally if $\Lambda$ is replaced with any other topologically nontrivial Legendrian knot with $U \prec \Lambda$. 
%!TEX root = concordances.tex

\section{A Census of Lagrangian Slice Knots}
\label{sec:census}

Recall that a Legendrian knot $\Lambda$ is said to be \emph{Lagrangian
  slice} if it bounds a Legendrian disk in $B^4$, or equivalently if
$U \prec \Lambda$ where $U$ is the standard unknot.
In this section we provide a complete list of nontrivial topological
knot types through 14 crossings that have Legendrian representatives
that are Lagrangian slice.
Our data support the following conjecture:

\begin{conjecture}
If $K$ is smoothly slice and $\maxtb(K)=-1$, then $K$ bounds a Lagrangian disk in $B^4$.
\label{conj:Lagr-slice}
\end{conjecture}

Note that the converse of this conjecture is true.

Our results are summarized in Table~\ref{table:census-diagrams}, which
provides a list of smoothly slice knots $K$ with up to $14$ crossings
such that $\maxtb(K) = -1$, along with a Legendrian representative
that bounds a Lagrangian disk; this table verifies Conjecture~\ref{conj:Lagr-slice} when $K$ has up to $14$ crossings.

In Section~\ref{ssec:alternating} and~\ref{ssec:nonalt}, we show
that nontrivial alternating knots cannot be Lagrangian slice and then
rule out all but $23$ nonalternating knot types of crossing numbers $\leq
14$.
In Section~\ref{ssec:table}, for each of the remaining $23$ knot types,
Table~\ref{table:census-diagrams} gives a Legendrian representative
$\Lambda$ with $U \prec \Lambda$, completing the census through $14$
crossings. We conclude by discussing in Section~\ref{ssec:nonrev} the
extent to which we can establish $\Lambda \not\prec U$ for these knots.

\subsection{Alternating knots}
\label{ssec:alternating}

We will first show that no Lagrangian slice knot can be alternating unless it is an unknot, using the following description of $\maxtb(K)$ for alternating $K$.

\begin{proposition}[{\cite[p.\ 1646]{ng-khovanov}}]
\label{prop:compute-maxtb}
Let $L$ be an oriented, alternating, nonsplit link, and $D$ a reduced alternating diagram for $L$ with $n_-(D)$ negative crossings.  Then
\[ \maxtb(L) = \sigma(L) - n_-(D) - 1, \]
where $\sigma(L)$ denotes signature, normalized so that the right-handed trefoil has signature $+2$.
\end{proposition}

\begin{theorem}
\label{thm:alternating-positive}
An alternating knot $K$ bounds an oriented Lagrangian surface $\Sigma$ in the standard symplectic $4$-ball if and only if $K$ is also a positive knot, in which case its Seifert, smooth, and topological 4-ball genera and the genus of $\Sigma$ are all equal to $\frac{\sigma(K)}{2}$.
\end{theorem}

\begin{proof}
The ``if'' direction is due to Hayden and Sabloff \cite{hayden-sabloff}, who showed that all positive links bound exact Lagrangian surfaces.  Conversely, suppose that $K$ bounds a Lagrangian surface $\Sigma$.  Then Chantraine \cite{chantraine-concordance} showed that $\maxtb(K)=2g(\Sigma)-1$ and $g_s(K)=g(\Sigma)$, where $g_s$ denotes the smooth slice genus.  Now Proposition~\ref{prop:compute-maxtb} tells us that
\[ 2g(\Sigma)-1 = \maxtb(K) = \sigma(K) - n_-(D) - 1 \]
for a reduced alternating diagram $D$ of $K$, hence $\sigma(K) = 2g_s(K) + n_-(D)$.  We apply Murasugi's bound $|\sigma(K)| \leq 2g_4(K)$ \cite{murasugi}, with $g_4(K)$ the topological $4$-ball genus, to get
\[ 2g_s(K) + n_-(D) = \sigma(K) \leq 2g_4(K) \leq 2g_s(K) \]
and hence each inequality is an equality, implying that $n_-(D)=0$ and $g_s(K) = g_4(K) = g(\Sigma)$ as desired.  The claim that $g(K) = g_s(K)$ follows because $K$ is positive \cite{rasmussen}.
\end{proof}

This gives a new, contact geometric proof of the result of Nakamura
\cite{nakamura} that a reduced alternating diagram of a positive knot
has no negative crossings.

\begin{corollary}
Nontrivial alternating knots are not Lagrangian slice.
\label{cor:alt}
\end{corollary}

\begin{proof}
By Theorem~\ref{thm:alternating-positive}, if $K$ is alternating and
Lagrangian slice, then $g(K)=0$.
\end{proof}

%%%%%%%%%%%%%%%%%%%%%%%%%

\subsection{Knots with at most 14 crossings}
\label{ssec:nonalt}

By Corollary~\ref{cor:alt}, to enumerate Lagrangian slice knots, it
suffices to restrict to nonalternating knots. Here we narrow the list
of candidates with up to $14$ crossings to a list of $23$, which we
then show are all Lagrangian slice in
Table~\ref{table:census-diagrams} in Section~\ref{ssec:table}.

\subsubsection{Up to 12 crossings}

If a Legendrian knot $\Lambda$ of smooth type $K$ bounds a Lagrangian disk, then it must have $\tb=-1$ \cite{chantraine-concordance} and be smoothly slice, and the latter implies $\maxtb(K) = -1$ by the slice-Bennequin inequality $\tb(\Lambda)+|r(\Lambda)| \leq 2g_s(K)-1$ \cite{rudolph-slice-bennequin}.  According to KnotInfo \cite{knotinfo}, there are exactly six nontrivial knot types of at most 12 crossings which are smoothly slice and have $\maxtb=-1$, namely
\[ \overline{9_{46}}, \overline{10_{140}}, 11n_{139}, 12n_{582}, \overline{12n_{768}}, 12n_{838}. \]
all of which are Lagrangian slice.

\begin{remark}
A note on chirality: We use $\overline{K}$ to denote the mirror of
$K$. There is some discrepancy in the literature over which of each
mirror pair is specified by a particular numbered knot. Our
conventions align with the \texttt{KnotTheory\`{}} package, available from
the Knot Atlas \cite{knotatlas}, which provides the Rolfsen knot
tables for knots through 10 crossings and the Hoste--Thistlethwaite
enumeration for 11 crossings and above. We note that this sometimes
differs from KnotInfo; for instance, by the
Thurston--Bennequin data in KnotInfo, $9_{46}$ rather than
$\overline{9_{46}}$ has $\maxtb=-1$.
\label{rmk:chirality}
It may be helpful to note that the Lagrangian slice knots which we have labeled $\overline{9_{46}},\overline{10_{140}}, 11n_{139}, 12n_{582}$ are unambiguously described as the $P(-m,-3,3)$ pretzel knots for $m=3,4,5,6$.
\end{remark}

%%%%%%%%%%%%%%%%%%%%%%%%%

\subsubsection{13- and 14-crossing knots}

We searched in \texttt{KnotTheory\`{}} for nonalternating 13- and 14-crossing knots satisfying the conditions of the following lemma.

\begin{lemma}
\label{lem:lagrangian-slice-criteria}
Suppose that $K$ admits a Legendrian representative $\Lambda$ which is Lagrangian slice.  If $\Delta_K(t)$, $P_K(a,z)$, $F_K(a,z)$, and $Kh_K(q,t)$ denote the Alexander, HOMFLY-PT, and Kauffman (Dubrovnik) polynomials of $K$ and the Poincar{\'e} polynomial of its Khovanov homology over $\Q$, respectively, then:
\begin{enumerate}
\item \label{item:top-slice} $\det(K)$ is a perfect square, the signature $\sigma(K)$ is zero, and $\Delta_K(t) = f(t)f(t^{-1})$ for some polynomial $f$.
\item \label{item:homfly-kauffman-bounds-eq} We have $\maxdeg_a P_K(a,z) = \maxdeg_a F_K(a,z) = 0$.
\item \label{item:ruling-polynomials-positive} If $p(z)$ and $f(z)$ denote the $a^0$-coefficients of $P_K(a,z)$ and $F_K(a,z)$, then $f(z) \geq p(z) \geq 0$, i.e.\ the coefficients of $f-p$ and of $p$ are all nonnegative.
\item \label{item:weak-khovanov--1} We have $\mindeg_q Kh_K(q,t/q) \geq -1$.
\item \label{item:s-invariant} If the Khovanov homology of $K$ has width at most $3$, then $Kh_K(q,-q^{-4}) = q+q^{-1}$.
\end{enumerate}
\end{lemma}

\begin{proof}
We know that $K$ must be smoothly slice, that $\Lambda$ admits a graded ruling, and that $\maxtb(K) = \tb(\Lambda) = -1$.  Thus \eqref{item:top-slice} follows from $K$ being topologically slice, and in particular $\Delta_K(t)=f(t)f(t^{-1})$ is the Fox--Milnor condition \cite{fox-milnor} and $\det(K)=f(-1)^2$.  Item \eqref{item:homfly-kauffman-bounds-eq} follows from achieving equality in the HOMFLY-PT \cite{franks-williams, morton} and Kauffman \cite{rudolph-kauffman} bounds on $\maxtb(K)$, namely $\maxtb(K) \leq -\maxdeg_a P_K(a,z)-1$ and likewise for $F_K(a,z)$, because Rutherford \cite{rutherford} showed that this equality follows from $\Lambda$ admitting a ruling.  Indeed, the leading coefficients $p(z)$ and $f(z)$ are then the 2-graded and ungraded ruling polynomials of $\Lambda$, and every 2-graded ruling is an ungraded ruling, so \eqref{item:ruling-polynomials-positive} is an immediate consequence.

Item \eqref{item:weak-khovanov--1} is the weak Khovanov bound $\maxtb(K) \leq \mindeg_q Kh_K(q,t/q)$ of \cite{ng-khovanov}.  Finally, we observe that \cite[Proposition~5.3]{rasmussen} says that knots of width at most $3$ satisfy
\[ Kh_K(q,t) = q^{s(K)}(q+q^{-1}) + (1+tq^4)Q_K \]
for some polynomial $Q_K$, and $s(K)=0$ since $K$ is smoothly slice, which implies \eqref{item:s-invariant}.
\end{proof}

There are six 13-crossing knot types satisfying all of the conditions of Lemma~\ref{lem:lagrangian-slice-criteria}, namely
\[ \overline{13n_{579}}, \overline{13n_{3158}}, 13n_{3523}, \overline{13n_{4236}}, 13n_{4514}, \overline{13n_{4659}}; \]
and 17 such knot types with 14 crossings, namely $14n_m$ for $m$ in
\begin{align*}
\overline{{2459}}, {2601}, 8091, \overline{{8579}}, \overline{9271}, \overline{{12406}}, \overline{{14251}}, 14799, 15489, \\
\overline{{15581}}, 17376, {18212}, {21563}, \overline{{22150}}, \overline{{22789}}, \overline{{24246}}, {25967}.
\end{align*}
Six of these knot types actually have $\maxtb(K) < -1$; we prove this by noting that they do not have arc index at most 10 \cite{jin-10} or 11 \cite{jin-park} and using Gridlink \cite{culler-gridlink} to find grid diagrams of complexity 12 for each of them, which must then be minimal.  We then apply Dynnikov and Prasolov's theorem that grid diagrams which minimize complexity maximize the Thurston--Bennequin invariant within a given knot type \cite{dynnikov-prasolov}. Specifically, the knot types
\[ \overline{13n_{4236}}, 14n_{8091}, \overline{14n_{9271}}, 14n_{14799}, 14n_{15489}, 14n_{17376} \]
all have arc index $12$, and in each case we find that $\maxtb(K) = -2$.

The remaining knot types are
\begin{align*}
& \overline{13n_{579}}, \overline{13n_{3158}}, 13n_{3523}, 13n_{4514}, \overline{13n_{4659}}, \\
& \overline{14n_{2459}}, 14n_{2601}, \overline{14n_{8579}}, \overline{14n_{12406}}, \overline{14n_{14251}}, \overline{14n_{15581}}, \\
& 14n_{18212}, 14n_{21563}, \overline{14n_{22150}}, \overline{14n_{22789}}, \overline{14n_{24246}}, 14n_{25967},
\end{align*}
and in 
Section~\ref{ssec:table}
we verify that these are all Lagrangian slice.

\begin{remark}
Applying Lemma~\ref{lem:lagrangian-slice-criteria}, we determine that there are at most 48 Lagrangian slice knot types with 15 crossings, namely $15n_m$ for $m$ in
\begin{align*}
& \overline{1481}, \overline{11562}, 11847, 11848, 38594, 41697, \overline{43982}, 46734, \overline{46855}, \overline{57450}, \\
& 73973, 77224, \overline{77245}, 77461, 81490, 83506, \overline{83742}, \overline{88825}, 96161, \overline{96452}, \\
& \overline{96790}, 103488, \overline{104659}, \overline{110305}, 110461, 112479, 127845, 127852, \overline{130682}, 131344, \\
& 132539, 133913, 134517, 135516, 136561, 138242, 138810, 139311, \overline{144052}, 145082, \\
& 153611, 153975, 154694, 155137, 155659, 155828, \overline{162371}, \overline{164338}.
\end{align*}
(The Khovanov homology of $15n_{115646}$ has width 4 but satisfies the weaker condition of \cite[Proposition~2.1]{shumakovitch-s}, so we can still use the proof of item \eqref{item:s-invariant} to show that $s(15n_{115646})=2$ and thus rule it out.)  However, we do not claim that all of these actually are Lagrangian slice.
\end{remark}

%%%%%%%%%%%%%%%%%%%%%%%%%

%%%%%%%%%%%%%%%%%%%%%%%%%
\subsection{The census}
\label{ssec:table}

For each of the $23$ knot types $K$ described in the previous
subsection, Table~\ref{table:census-diagrams} provides one Legendrian
knot $\Lambda$ (depicted as a grid diagram), marked to indicate where
one can perform surgery (a pinch move) to
construct a 
Lagrangian concordance $U \prec \Lambda$.
We also note
the cases in which we can prove that $\Lambda \not\prec U$; see
Section~\ref{ssec:nonrev} for discussion.

\newcommand{\knotstr}[3]{
\ifstrequal{#1}{m}
{\overline{\ifnumcomp{#2}{<}{11}{#2_{#3}}{#2n_{#3}}}}
{\ifnumcomp{#2}{<}{11}{#2_{#3}}{#2n_{#3}}}
}

\newcommand{\knotrow}[6]{
$\knotstr{#1}{#2}{#3}$ &
\includegraphics[scale=0.4]{knots/#2n#3} &
\parbox{6cm}{$X: #4$ \\ $\mathrlap{O}{\phantom{X}}: #5$} &
\ifstrequal{#6}{y}{\checkmark}{}
}

\begin{longtable}{lm{2cm}@{\hskip 1.5cm}l@{\hskip 1cm}c}
$K$ & $\Lambda$ & Coordinates & $\Lambda \not\prec U$? \\
\toprule
\endhead
\knotrow{m}{9}{46}{1,6,7,5,3,4,2,8}{5,2,4,8,6,1,7,3}{y}
\\ \midrule
\knotrow{m}{10}{140}{1,2,8,6,7,4,5,3,9}{6,7,4,9,3,1,2,8,5}{}
\\ \midrule
\knotrow{}{11}{139}{1,8,9,7,3,4,2,6,5,10}{7,2,4,10,8,1,5,3,9,6}{}
\\ \midrule
\knotrow{}{12}{582}{1,2,10,8,9,4,5,3,7,6,11}{8,9,4,11,3,1,2,6,5,10,7}{}
\\ \midrule
\knotrow{m}{12}{768}{2,3,10,1,9,6,7,5,8,4}{6,9,7,8,5,10,4,2,3,1}{y}
\\ \midrule
\knotrow{}{12}{838}{8,6,9,7,4,5,2,10,1,3}{5,2,3,10,8,1,9,4,6,7}{y}
\\ \midrule

\knotrow{m}{13}{579}{9,7,8,6,2,11,1,10,4,5,3}{4,10,1,9,7,5,6,3,11,2,8}{y}
\\ \midrule
\knotrow{m}{13}{3158}{11,9,10,8,6,7,5,3,4,1,2}{3,1,7,2,9,4,8,6,11,5,10}{y}
\\ \midrule
\knotrow{}{13}{3523}{1,10,11,9,3,4,2,6,5,8,7,12}{9,2,4,12,10,1,5,3,7,6,11,8}{}
\\ \midrule
\knotrow{}{13}{4514}{2,10,11,4,9,8,6,1,7,3,5}{9,6,7,8,5,10,3,4,2,11,1}{}
\\ \midrule
\knotrow{m}{13}{4659}{4,8,5,10,3,9,11,7,2,1,6}{11,2,9,7,8,4,6,1,10,5,3}{}
\\ \midrule

\knotrow{m}{14}{2459}{5,7,1,11,12,9,4,3,2,6,10,8}{9,12,10,2,8,6,7,1,5,3,4,11}{}
\\ \midrule
\knotrow{}{14}{2601}{2,9,7,10,5,8,4,6,1,3,11}{8,6,1,2,11,3,9,10,5,7,4}{y}
\\ \midrule
\knotrow{m}{14}{8579}{8,5,10,2,9,6,7,4,1,3,11}{1,9,7,8,3,11,2,10,5,6,4}{y}
\\ \midrule
\knotrow{m}{14}{12406}{8,6,11,2,10,1,3,9,5,7,4}{1,10,7,9,5,6,8,4,2,3,11}{y}
\\ \midrule
\knotrow{m}{14}{14251}{10,4,3,9,5,7,8,6,2,11,1}{5,2,8,4,11,10,1,9,7,3,6}{}
\\ \midrule
\knotrow{m}{14}{15581}{1,9,8,5,11,7,10,4,6,2,3}{4,2,10,9,6,3,5,8,1,7,11}{y}
\\ \midrule

\knotrow{}{14}{18212}{1,2,12,10,11,4,5,3,7,6,9,8,13}{10,11,4,13,3,1,2,6,5,8,7,12,9}{}

\\ \midrule
\knotrow{}{14}{21563}{4,1,10,11,9,6,7,8,5,2,3}{11,7,5,8,2,10,4,3,1,6,9}{y}
\\ \midrule
\knotrow{m}{14}{22150}{10,6,8,5,12,7,1,9,11,4,2,3}{4,1,2,11,9,10,8,6,3,12,5,7}{y}
\\ \midrule
\knotrow{m}{14}{22789}{2,12,6,11,8,4,10,9,3,7,5,1}{11,8,10,7,3,1,2,5,6,4,12,9}{}
\\ \midrule
\knotrow{m}{14}{24246}{11,9,10,7,8,6,2,1,5,4,12,3}{8,12,4,11,5,1,7,3,9,2,6,10}{}
\\ \midrule
\knotrow{}{14}{25967}{9,1,12,5,8,6,7,4,2,3,10,11}{12,10,2,1,11,9,3,8,5,6,4,7}{}
\\ \bottomrule
\caption{Lagrangian slice knots through 14 crossings.}
\label{table:census-diagrams}
\end{longtable}

%%%%%%%%%%%%%%%%%%%%%%%%%
\subsection{Non-reversible concordances}
\label{ssec:nonrev}

Here we study the question of which concordances $U \prec
\Lambda$ from the census in Section~\ref{ssec:table} are
non-reversible in the sense that $\Lambda \not\prec U$,
i.e., there is a concordance from $U$ to $\Lambda$ but not vice versa.
Our results are summarized in Table~\ref{table:census-rulings} below.

An obstruction to the existence of a concordance giving $\Lambda \prec U$
is provided by Corollary~\ref{cor:ruling-polynomial-uku}. For $6$ of
the $23$ Legendrian knots $\Lambda$ from Table~\ref{table:census-diagrams}
(of types
$\overline{9_{46}}$, $12n_{838}$, $\overline{13n_{3158}}$,
$14n_{2601}$, $\overline{14n_{15581}}$, $14n_{21563}$),
  the ungraded ruling polynomial $R^1_\Lambda(z)$ is not equal to $1$,
  and so $\Lambda \not\prec U$ by
  Corollary~\ref{cor:ruling-polynomial-uku}.
  In $5$ additional cases
  ($\overline{12n_{768}}$, $\overline{13n_{579}}$,
  $\overline{14n_{8579}}$, $\overline{14n_{12406}}$, $\overline{14n_{22150}}$),
  we may similarly apply Theorem~\ref{thm:twn-ruling-polynomials} to
  the $2$-cable $\Lambda_2$ of $\Lambda$:
  the ungraded ruling polynomial 
  for $\Lambda_2$ is
  not equal to $z^{-1}$, and hence
  Theorem~\ref{thm:twn-ruling-polynomials} shows that
  $\Lambda\not\prec U$.
Note that this argument shows $\Lambda\not\prec U$ for all Legendrian
representatives $\Lambda$ of these knot types, not just the particular
ones from Table~\ref{table:census-diagrams}; see Remark~\ref{rmk:smoothtype}.

\newcommand{\knotrulings}[7]{$\knotstr{#1}{#2}{#3}$ &
%$#4$ & $#5$ &
$#6$ & $#7$}

\begin{table}
\begin{centering}
\begin{tabular}{r|r|l}
%{r|r|l|r|l}
$K$ &
% $n$ & $z^{n-1}R_{K_n}(z)$ (graded) &
$n$ & $z^{n-1}R^1_{\Lambda_n}(z)$ \\%(ungraded) \\
\hline
\knotrulings{m}{9}{46}{1}{2}{1}{2} \\
\knotrulings{m}{10}{140}{>4}{}{>2}{} \\
\knotrulings{}{11}{139}{>3}{}{>2}{} \\
\knotrulings{}{12}{582}{>3}{}{>2}{} \\
\knotrulings{m}{12}{768}{>3}{}{2}{z^{12} + 12z^{10} + 49z^8 + 78z^6 + 41z^4 + 4z^2 + 1} \\
\knotrulings{}{12}{838}{1}{2}{1}{2} \\
\knotrulings{m}{13}{579}{>4}{}{2}{z^{12} + 9z^{10} + 25z^8 + 21z^6 + 4z^4 + 1} \\
\knotrulings{m}{13}{3158}{1}{z^2+3}{1}{z^4+3z^2+3} \\
\knotrulings{}{13}{3523}{>4}{}{>2}{} \\
\knotrulings{}{13}{4514}{>3}{}{>2}{} \\
\knotrulings{m}{13}{4659}{>4}{}{>2}{} \\

\knotrulings{m}{14}{2459}{>4}{}{>2}{} \\
\knotrulings{}{14}{2601}{1}{z^2+3}{1}{z^4+3z^2+3} \\
\knotrulings{m}{14}{8579}{>3}{}{2}{3z^{12} + 27z^{10} + 81z^8 + 93z^6 + 38z^4 + 4z^2 + 1} \\
\knotrulings{m}{14}{12406}{>3}{}{2}{3z^{12} + 26z^{10} + 72z^8 + 68z^6 + 17z^4 + 1} \\
\knotrulings{m}{14}{14251}{>3}{}{>2}{} \\
\knotrulings{m}{14}{15581}{1}{2}{1}{z^4+2z^2+2} \\
\knotrulings{}{14}{18212}{>3}{}{>2}{} \\
\knotrulings{}{14}{21563}{1}{z^2+2}{1}{z^4+3z^2+2} \\
\knotrulings{m}{14}{22150}{>4}{}{2}{z^{12}+8z^{10}+18z^8+8z^6+1} \\
\knotrulings{m}{14}{22789}{>3}{}{>2}{} \\
\knotrulings{m}{14}{24246}{>4}{}{>2}{} \\
\knotrulings{}{14}{25967}{>3}{}{>2}{} \\
\end{tabular}
\end{centering}

\caption{For each Lagrangian slice knot $K$ through $14$ crossings with Legendrian representative $\Lambda$ from Table~\ref{table:census-diagrams},
the least $n$ for which the ruling polynomial of the $n$-cable of
$\Lambda$ is known to differ from $1$.  An entry of the form ``$n > 2$''
indicates that $z^{n-1}R^1_{\Lambda_n}(z) = 1$ for all $n \leq
2$. 
}
\label{table:census-rulings}
\end{table}

We have not been able to rule out the possibility that $\Lambda\prec
U$ for $\Lambda$ representing the remaining $12$ knot types. One indication that it may be
difficult to do so for at least some of these knots is the curious
family of Legendrian $P(-m,-3,3)$ pretzel knots for $m \geq 4$, of which
$\overline{10_{140}}$, $11n_{139}$, $12n_{582}$, $13n_{3523}$,
$14n_{18212}$ are shown in
Table~\ref{table:census-diagrams}, with a natural generalization to
all $m \geq 4$.
For this family of Legendrian knots, it can be shown
that the Legendrian contact homology DGA is stable tame isomorphic to
the DGA for the unknot $U$ (even over $\zz[t,t^{-1}]$), and so these
knots are indistinguishable from $U$ from the viewpoint of contact
homology.
In fact, it follows from this fact that any $n$-cable of these knots
has the same ungraded ruling polynomial as the $n$-cable of the
unknot, though we omit the proof.

%%%%%%%%%%%%%%%%%%%%%%%%%

%%%%%%%%%%%%%%%%%%%%%%%%%

\bibliographystyle{hplain}
\bibliography{References}

\end{document}